\documentclass[11pt, reqno]{amsart}
\usepackage{version}
\usepackage{amssymb}
\newtheorem{theorem}{Theorem}[section]
\newtheorem{lemma}[theorem]{Lemma}

\theoremstyle{definition}

\newtheorem{remark}[theorem]{Remark}

\numberwithin{equation}{section}

\begin{document}
\title[]{Extensions to the boundary of Riemann maps on varying domains in the complex plane}

\author{Jan Pel}
\author{Han Peters}
\author{Erlend F. Wold}

\address{E. F. Wold: Department of Mathematics\\
University of Oslo\\
Postboks 1053 Blindern, NO-0316 Oslo, Norway}\email{erlendfw@math.uio.no}

\address{H. Peters: Korteweg de Vries Institute for Mathematics\\
University of Amsterdam \\
P.O. Box 94248 \\
1090 GE Amsterdam, Netherlands}\email{h.peters@uva.nl}
%
%
\date{\today}


\maketitle

\begin{abstract}
We give a short proof of the convergence to the boundary of Riemann maps on varying domains.
Our proof provides a uniform approach to several ad-hoc constructions that have recently appeared in the literature.
\end{abstract}

\section{Introduction}

In the following $\triangle$ denotes the unit disk in $\mathbb C$, and $\triangle_r(z_0)$ denotes
the disk of radius $r$ centred at the point $z_0$.

\begin{theorem}\label{main}
Let $\Gamma\subset\partial\triangle$ be an open interval, let $U\subset\mathbb C$ be an
open neighbourhood of $\overline\Gamma$, and suppose that $\{D_j\}_{j=1}^\infty$ is
a sequence of simply connected domains containing $\triangle$, each bounded by a closed Jordan curve
and contained in $\triangle_R$ for some fixed $R>0$,
and assume that $D_j\cap U=\triangle\cap U$ for all $j$.   Assume further
that $D_j\rightarrow\triangle$ with respect to kernel convergence, and
for each $j$ let $f_j:D_j\rightarrow\triangle$ be the Riemann map with
$f_j(0)=0$ and $f_j'(0)>0$.   Then $f_j\rightarrow\mathrm{id}$
uniformly on compact subsets of $\triangle\cup\Gamma$.
\end{theorem}

A special case of Theorem \ref{main} appeared in \cite{ForstnericWold2009} as an ingredient in the proof of the fact
that any bordered Riemann surface in $\mathbb C^2$ embeds properly into $\mathbb C^2$.  It appeared again later in \cite{DiederichFornaessWold} and \cite{DengFornaessWold} as ingredients in techniques for globally exposing points
on certain Stein compacts in $\mathbb C^n$ and in Stein spaces,  and recently yet another improvement appeared
in \cite{Forstneric}, as an ingredient in the construction of certain exotic proper embeddings of the unit disk into $\mathbb B^2$.
However, in all cases, the sequences $D_j$ are ad hoc constructions intended for very specific purposes, and
all of them are obtained by adding to $D$ certain shrinking tubes around special arcs attached to $D$.   The proofs in these special cases do not adapt to the general setting of Theorem \ref{main}.

\begin{remark}
This result holds in much more generality.   For instance, if $\triangle$ in the theorem is replaced by
any simply connected Jordan domain $D$, and if $\Gamma\subset\partial D$ is a Jordan arc, elementary
arguments involving Riemann mappings reduce the situation to the setting of the theorem.   Inspecting the
proof below also makes it clear that one may consider much more general domains, whose boundaries
only contain a common Jordan arc, and without the assumption imposed by $U$, but one might have to consider
one sided convergence. One may also prove similar results for domains $D_j$ which are not simply connected, but which
are all conformally equivalent to a given domain $D$. By introducing additional arguments one
can get convergence in $C^k$-norm, granted that the arc $\Gamma$ (in the general case) is of
class $C^{k,\alpha}$ for some $\alpha>0$.
\end{remark}

\section{Proof of Theorem \ref{main}}

It is well known that each $f_j$ extends to a homeomorphism between the closures of the domains.  We claim first that for any compact
subset $K\subset\Gamma$ the family $f_j|_K$ is equicontinuous.   If not, then by passing to a subsequence we may assume that there
is a sequence of shrinking intervals $\gamma_j\subset\Gamma$ with end points $a_j,b_j \rightarrow e^{i\theta}$ for which $f_j(\gamma_j)\subset\partial\triangle$
has length greater than some fixed $\epsilon>0$ for all $j$. Consider the sequence of inverse maps $g_j=f_j^{-1}$.
After composing with a rotation, we may assume that $f_j(\gamma_j)\supset\{e^{i\alpha}:-\delta<\alpha<\delta\}=:\eta$
for some $\delta>0$.  Choose $0<x_0<1$ and $r>0$ such that for all $x_0<x<1$ we have that
$\triangle_r(x)\cap\partial\triangle\subset\eta$.
Now apply Lemma \ref{Lind} below to the bounded functions $g_j - e^{i\theta}$, defined on the unit disk $D = \triangle$ and with $z_0 = x$. It follows that $g_j(x)-e^{i\theta}\rightarrow 0$ for all $x$, from which it follows that $g_j \rightarrow e^{i\theta}$ uniformly on compacts in $\triangle$.
However, by kernel convergence $g_j\rightarrow \mathrm{id}$ uniformly on compacts in $\triangle$, giving a contradiction. \

After passing to a subsequence we may assume that $f_j|_K$ is a Cauchy sequence.
By an argument similar to the one just given,
it follows from Lemma \ref{Lind} that $f_j$ is a Cauchy sequence on an extension of $K$ into $\triangle$.
The result follows from the fact that $f_j\rightarrow\mathrm{id}$ uniformly on compact sets in $\triangle$.
$\hfill\square$

\section*{Lindel\"{o}fs maximum prinicple}

We recall with a proof the following result (see e.g. Goluzin \cite{Goluzin}, page 33).

\begin{lemma}(Lindel\"{o}f)\label{Lind}
Let $D\subset\mathbb C$ be a domain, let $z_0\in D$, and suppose for $r>0$ and $m\in\mathbb N$,
there is an interval $I\subset \partial\triangle_r(z_0)$ of length at least $\frac{2\pi r}{m}$, such that $I\subset\mathbb C\setminus\overline D$.  Then for any $f\in\mathcal O(D)$ with $|f(z)|\leq M$ for all $z\in D$,
satisfying
$$
\limsup_{(\triangle_r(z_0)\cap D)\ni z\rightarrow \partial D} |f(z)|<\epsilon,
$$
we have that $|f(z_0)|<(\epsilon M^{m-1})^{1/m}$.
\end{lemma}

\begin{proof}
Without loss of generality we may assume that $z_0=0$.  Consider the domain
\begin{equation}
\tilde D := \bigcap_{k=0}^{m-1}(e^{2\pi ik/m}\cdot D),
\end{equation}
the intersection of $m$ rotated copies of $D$. Then $\tilde D$ is invariant under rotation by the angle $2\pi/m$,
and is strictly contained in $\triangle_r(0)$. Define the function $\tilde f$ on $\tilde D$ by
\begin{equation}
\tilde f(z) := \prod_{k=0}^{m-1} f(e^{2\pi ik/m}\cdot z).
\end{equation}
Then we have that
\begin{equation}
\underset{\tilde D\ni z \rightarrow\partial\tilde D}{\limsup}|f(z)|\leq M^{m-1}\cdot\epsilon,
\end{equation}
and so by the maximum principle, $|\tilde f(0)|=|f(0)|^m\leq M^{m-1}\cdot\epsilon$.
\end{proof}

\bibliographystyle{amsplain}

\end{document}